\documentclass{amsart}
\usepackage{amsfonts,amssymb,amscd,amsmath,enumerate,verbatim,calc}
\usepackage[all]{xy}

\newcommand{\Cc}{\mathcal{C} }

\newcommand{\m}{\mathfrak{m} }

\newcommand{\K}{\mathcal{K} }
\newcommand{\D}{\mathcal{D} }
\newcommand{\T}{\mathcal{T} }
\newcommand{\C}{\mathcal{C} }
\newcommand{\A}{\mathcal{A} }
\newcommand{\W}{\mathcal{W} }
\newcommand{\Z}{\mathbb{Z} }
\newcommand{\Xb}{\mathbf{X}_\bullet}
\newcommand{\Zb}{\mathbf{Z}_\bullet}
\newcommand{\Yb}{\mathbf{Y}_\bullet}
\newcommand{\Kb}{\mathbf{K}_\bullet}

\newcommand{\Pb}{\mathbf{P}_\bullet}
\newcommand{\Qb}{\mathbf{Q}_\bullet}
\newcommand{\Fb}{\mathbf{F}_\bullet}
\newcommand{\Wb}{\mathbf{W}_\bullet}

\newcommand{\rt}{\rightarrow}

\newcommand{\ctensor}{\widehat{\otimes}}

\newcommand{\image}{\operatorname{image}}

\newcommand{\cone}{\operatorname{cone}}

\newcommand{\Mod}{\operatorname{Mod}}
\newcommand{\mmod}{\operatorname{mod}}
\newcommand{\rank}{\operatorname{rank}}

\newcommand{\proj}{\operatorname{proj}}

\newcommand{\Proj}{\operatorname{Proj}}

\newcommand{\Hom}{\operatorname{Hom}}
\newcommand{\cHom}{\operatorname{\mathcal{H}om}^c}

\newcommand{\End}{\operatorname{End}}

\theoremstyle{plain}

\newtheorem{theorem}{Theorem}[section]
\newtheorem{corollary}[theorem]{Corollary}
\newtheorem{lemma}[theorem]{Lemma}
\newtheorem{proposition}[theorem]{Proposition}

\theoremstyle{definition}
\newtheorem{definition}[theorem]{Definition}
\newtheorem{s}[theorem]{}
\newtheorem{remark}[theorem]{Remark}

\theoremstyle{remark}

\begin{document}

\title[$2$-Periodic complexes]{$2$-Periodic complexes over regular local rings}
\author{Tony~J.~Puthenpurakal}
\date{\today}
\address{Department of Mathematics, IIT Bombay, Powai, Mumbai 400 076}

\email{tputhen@math.iitb.ac.in}
\subjclass{Primary  13D09,  16G70 ; Secondary 13H05, 13H10}
\keywords{$2$-periodic complexes, its homotopy category and bounded derived category, Auslander-Reiten triangles}

 \begin{abstract}
Let $(A,\mathfrak{m})$ be a  regular local ring of dimension $d \geq 1$. Let $\mathcal{D}^2_{fg}(A)$ denote the derived category of $2$-periodic complexes  with   finitely generated cohomology modules.
Let $\mathcal{K}^2(\proj A) $ denote the homotopy category of $2$-periodic complexes of finitely generated free $A$-modules.
 We show the natural  map $\mathcal{K}^2(\proj A) \rt \mathcal{D}^2(A)$ is an equivalence of categories. When $A$ is complete we show that $\mathcal{K}^2_f(\proj A)$ ($2$-periodic complexes with finite length cohomology) is Krull-Schmidt with   Auslander-Reiten (AR) triangles. We also compute the AR-quiver of  $\mathcal{K}^2_f(\proj A)$ when $\dim A = 1$.
\end{abstract}
 \maketitle
\section{introduction}
Let $\Gamma$ be an Artin algebra. The study of $m$-periodic  complexes over $\Gamma$ is well studied.
See \cite{Br}, \cite{CC}, \cite{Chen}, \cite{Chen-Deng}, \cite{S1} and \cite{S2}.
As exemplified by Auslander that concepts in the study of representation
theory of Artin algebras have natural analogues in study of  Cohen-Macaulay local rings. In this paper we study
$2$-periodic complexes over  commutative Noetherian
rings.

Let $(A, \m)$ be a Noetherian local ring.
We index complexes of $A$-modules cohomologically
$$\Xb \colon  \cdots \rt \Xb^{n-1} \xrightarrow{\partial_{n-1}} \Xb^n \xrightarrow{\partial_{n}} \Xb^{n+1} \rt \cdots.$$
By a $2$-periodic complex we mean a complex  $\Xb$ with $\Xb^{2n + 1} =\Xb^1$, $\Xb^{2n} =  \Xb^0$, $\partial_{2n + 1} = \partial_1$ and $\partial_{2n} = \partial_0$ for all $n \in \Z$.
Let $\A$ denote a full additive subcategory of $\Mod(A)$ (for most of our applications $\A$ will be $\Mod(A)$, $\mmod(A)$ or
$\proj(A)$).
One can define a category $\C^2(\A)$ of $2$-periodic $\A$-complexes with $2$-periodic chain maps (see \ref{senorita} for definition). It is readily verified that $\C^2(\A)$ is an additive category (and abelian if $\A$ is). We can define $2$-periodic chain-homotopy (see \ref{rachel} for definition) and define the $2$-periodic homotopy category $\K^2(\A)$. It is readily proved that $\K^2(A)$ is a triangulated category. By inverting quasi-isomorphisms one obtains the $2$-periodic derived category, $\D^2(\A)$. Our interest is to study $\D^2_{fg}(A) = \D^2_{fg}(\Mod A)$, the thick subcategory of $\D^2(\Mod A)$ consisting of $2$-periodic complexes $\Xb$ with $H^i(\Xb)$ a finitely generated $A$-module for all $i \in \Z$.

In general one cannot say much about $\D^2_{fg}(A)$. However the situation drastically improves when $A$ is regular local. Let $\K^2(\proj A) $ denote the homotopy category of $2$-periodic complexes of finitely generated free $A$-modules.
 We have an obvious exact functor $\phi \colon \K^2(\proj A) \rt \D^2_{fg}(A)$. We prove
 \begin{theorem}
 \label{first} Let $(A, \m)$ be regular local. Then $\phi$ is  an equivalence of triangulated categories.
 \end{theorem}

\begin{remark}
  Theorem \ref{first} also implies that the natural map $\psi \colon \K^2(\proj A) \rt \D^2(\mmod(A))$ is also an equivalence. However we need Theorem \ref{first} for our applications.
\end{remark}
 In view of Theorem \ref{first} it is a natural question on when $\K^2(\proj A)$ is a Krull-Schmidt category. We prove quite generally.
 \begin{theorem}
 \label{ks} Let $(A, \m)$ be a Henselian Noetherian local ring. Then $\K^2(\proj A)$ is a Krull-Schmidt category.
 \end{theorem}

 A natural question for a Krull-Schmidt triangulated category is whether it has Auslander-Reiten (AR) triangles. Let $\K^2_f(\proj A)$ denote the full subcategory of $\K^2(\proj A)$ consisting of complexes with finite length cohomology. We prove
 \begin{theorem}
 \label{ar} Let $(A, \m)$ be a complete regular local ring. Then $\K^2_f(\proj A)$ has AR-triangles.
 \end{theorem}

In the last section we explicitly compute the AR-quiver when $A$ is a complete DVR with algebraically closed residue field. It consists of two infinite tubes, see \ref{quiver}.
Thus this is an example of a triangulated category with infinitely many indecomposables but finitely many components of the AR-quiver.
In the famous book \cite[p.\ 409]{ARS} one of the conjectures is that if an artin algebra $\Gamma$ is of infinite representation type then its AR-quiver  has infinitely many components. Thus this conjecture does not hold more generally.

\begin{remark}
A natural question is whether analogues of Theorem \ref{first}, \ref{ar}  for $m$-periodic complexes. We believe that analogues results hold. However it is technically very challenging. So
in this paper we only discuss $2$-periodic complexes.
\end{remark}

We now describe in brief the contents of this paper. In section two we discuss some preliminary results on $\D(\Mod A)$ (the unbounded derived category of $A$) when $A$ is regular local.
In section three we discuss some preliminaries on $2$-periodic complexes that we need. In the next section we describe some constructions on two periodic complexes that we need. In section five we discuss minimal complexes in $\C^c(\proj A)$. In section six we discuss the Krull-Schmidt property of $\K^2(\proj A)$ when $A$ is henselian.  In the next section we prove a Lemma which is very useful to us. In section eight we show that any $2$-periodic complex $\Xb$  with finitely generated cohomology has a projective resolution $\Pb \rt \Xb$ with $\Pb \in \K^2(\proj A)$. In section nine we give a proof of Theorem \ref{first}. In section ten we discuss some preliminaries on AR-triangles and irreducible maps. In section eleven we give a proof of Theorem \ref{ar}. Finally
in the last section we explicitly compute the AR-quiver when $A$ is a complete DVR with algebraically closed residue field.
 \section{Some preliminaries on $\D(\Mod A)$ }
 Although we are primarliy interested in cyclic complexes we need a few prelimnary results on the (unbounded) derived category of modules (not necessarily finitely generated) over regular local rings.  \emph{In this section $(A,\m)$ is a regular local ring}. The results of this section are probably already known. We give a proof due to lack of a reference.

 \s
 Let $\Mod(A)$ denote the category of all $A$-modules. Let $\Proj A$ denote the category all projective $A$-modules. Note if $P$ is a projective $A$-module (not necessarily finitely generated) then it is a free $A$-module, \cite[2.5]{M}.
 Let $\C(\Mod A)$ denote the category  of  (co-chain) complexes of $A$-modules.
 We do not impose any boundedness condition on complexes in $\C(\Mod A)$.
 Let $\K(\Mod A)$ denote the homotopy category of complexes of $A$-modules and let $K(\Proj A)$  denote the homotopy categoy of projective $A$-modules. Let $\D(\Mod A)$ denote the  (unbounded) derived category of $A$-modules.

 \begin{lemma}
  \label{acyclic}
  Let $\Xb \in \K(\Proj A)$. If $\Xb$ is acyclic (i.e., $H^*(\Xb) = 0$) then $\Xb = 0$ in $\K(\Proj A)$.
 \end{lemma}
\begin{proof}
 Let the differential in $\Xb$ be given by $\partial_n \colon \Xb^n \rt \Xb^{n+1}$. Set $P_n = \image \partial_n$.
 As $\Xb$ is acyclic we have that $P_n$ is a $d^{th}$-syzygy of $P_{n+d}$ (here $d = \dim A$). So $P_n$ is a projective module.
 It is then readily proved that the identity map on $\Xb$ is null-homotopic. So $\Xb = 0$ in $\K(\Proj A)$.
\end{proof}

As a consequence we obtains
\begin{corollary}\label{q-proj}
Let $f \colon \Pb \rt \Qb$ be a quism where $\Pb, \Qb \in \K(\Proj A)$. Then $f$ is an isomorphism in $\K(\Proj A)$.
\end{corollary}
\begin{proof}
 We have a triangle $\Pb \xrightarrow{f} \Qb \rt C(f) \rt \Pb[1]$. As $f$  is a quism we get $H^*(C(f)) = 0$. So by Lemma \ref{acyclic} we get $C(f) = 0$ in $\K(\Proj A)$. Thus $f$ is an isomorphism in $\K(\Proj A)$.
\end{proof}

\begin{remark}
 Let $\Xb \in \C(\Mod A)$. Then there exists $\Pb \in \C(\Proj A)$ and a quism $p \colon \Pb \rt \Xb$, see \cite[5.1.7]{CFH}.
\end{remark}

 Next we show
 \begin{lemma}
  \label{right}
  Let  $\Pb \in \K(\Proj A)$ and let $\Xb \in \K(\Mod A)$. Assumme there is a quism $s \colon \Xb \rt \Pb$. Then there exists a quism $t \colon \Pb \rt \Xb$ such that $s\circ t = 1_{\Pb}$ in $\K(\Mod A)$.
 \end{lemma}
\begin{proof}
 Let $\eta \colon \Qb \rt \Xb$ be a quism with $\Qb \in \C(\Proj A)$. Then $g = s \circ \eta \colon \Qb \rt \Pb$ is a quism and so invertibe in $\K(\Proj A)$. Set $t = \eta  \circ g^{-1}$. The result follows.
\end{proof}

Finally we show
\begin{theorem}
 \label{p-d} Let $\Pb \in \K(\Proj A)$ and let $\Xb \in \K(\Mod A)$. Then the natural map
 \[
 \theta \colon  \Hom_{\K(\Mod A)}(\Pb, \Xb) \rt  \Hom_{\D(\Mod A)}(\Pb, \Xb)
 \]
is an isomorphism.
\end{theorem}
\begin{proof}
 Let $f \circ s^{-1} \colon \Pb \xleftarrow{s} \Zb \xrightarrow{f} \Xb$ be a left fraction. Note $s$ is a quism. By \ref{right} there exists a quism $t \colon  \Pb \rt \Zb$ such that $s\circ t = 1_{\Pb}$ in $\K(\Mod A)$. Then check that
  $f\circ s^{-1} = f\circ t$. Thus $\theta$ is surjective.

 Let $f \colon \Pb \rt \Xb$ be such that $\theta(f) = 0$. Then by
 \cite[2.1.26]{N} there exists a quism $s \colon \Zb \rt \Pb$ such that $f \circ s = 0$. By \ref{right} there exists a quism $t \colon  \Pb \rt \Zb$ such that $s\circ t = 1_{\Pb}$ in $\K(\Mod A)$.
 Then $f = f \circ s \circ t = 0$. So $\theta$ is injective.
\end{proof}
\section{some preliminaries on $2$-periodic complexes}
In this section $(A,\m)$ will denote a Noetherian ring. Let $\Mod(A)$ be the category of all $A$-modules. By $\mmod(A)$ we will denote the category of all finitely generated $A$-modules.

\s \label{senorita} Let $\A$ denote a full additive subcategory of $\Mod(A)$ (for most of our applications $\A$ will be $\Mod(A)$, $\mmod(A)$ or
$\proj(A)$). By an $\A$-$2$-periodic complex we mean an $\A$ (co-chain) complex $\Xb$ such that $\Xb^{2n} = \Xb^0$, $\Xb^{2n +1} = \Xb^1$ and differentials $\partial_{2n} = \partial_0$ and $\partial_{2n +1} = \partial_1$ for all $n \in \Z$.
Let $\Xb, \Yb$ be $2$-periodic $\A$-complexes. By a $2$-periodic map $\phi \colon \Xb \rt \Yb$ we mean a chain map $\phi$ with $\phi_{2n} = \phi_0$ and $\phi_{2n+1} = \phi_1$ for all $n \in \Z$.
Let $\C^2(\A)$ be the class of all $2$-periodic $\A$-complexes and $2$-periodic morphisms. Clearly $\C^2(\A)$ is an additive category. If $\A = \mmod(A)$ or $\Mod(A)$ then it is also clearly abelian. Let $\Xb$ be a $2$-periodic $\A$-complex. By $\Xb[1]$ we mean the $\A$-complex with $\Xb[1]^n = \Xb^{n+1}$ and $\partial_{\Xb[1], n} = - \partial_{\Xb, n+1}$.

\s \label{rachel} Let $\Xb, \Yb \in \C^2(\A)$ and let $\phi, \psi \colon \Xb \rt \Yb$ be $2$-periodic maps. We say $\phi, \psi$ are  $2$-homotopic if there exists maps $s_n \colon \Xb^n \rt \Yb^{n-1}$ with $s_{2n} = s_0$ and $s_{2n+1} = s_1$ for all $n \in \Z$ with $\phi - \psi = \partial_{\Yb}\circ s + s \circ \partial_{\Xb}.$ It can be easily seen that if $\phi, \psi \colon \Xb \rt \Yb$ are $2$-homotopic and if $u \colon \Pb \rt \Xb$ and $v \colon \Yb \rt \Qb$ are $2$-periodic maps then
$\phi \circ u, \psi \circ u$ (and $v \circ \phi$, $v \circ \psi$) are $2$-homotopic. So we may  form the $2$-periodic homotopy category $\K^2(\A)$.

\s Let $\Xb, \Yb$ be $2$-periodic $\A$-complexes and let $f \colon \Xb \rt \Yb$ be a $2$-periodic map.
We define $\cone(f)$ as follows: $\cone(f)^n = \Xb^{n+1} \oplus \Yb^n$ and differentials defined as
$$ \partial_n(x,y) = (-\partial_{\Xb, n+1}(x), \partial_{\Yb, n}(y) - f_{n+1}(x)).$$
Clearly
$\cone(f) \in \C^2(\A)$. We also have the usual exact sequence complexes
\[
 0 \rt \Yb \xrightarrow{u} \cone(f) \xrightarrow{v}  \Xb[1] \rt 0
 \]
where $u_n(y) = (0,y)$ and $v_n(x,y) = -x$. Clearly $u, v$ are $2$-periodic maps.

\s \emph{Triangles in $\K^2(\A)$:} By a strict triangle in $\K^2(\A)$ we mean
$$\Xb \xrightarrow{f} \Yb \xrightarrow{u} \cone(f) \xrightarrow{v} \Xb[1],$$
where $f \colon \Xb \rt \Yb$ is a $2$-periodic chain map and $u, v$ are as above. All triangles isomorphic to a strict triangle are exact triangles. With this triangulated structure a standard and routine argument yields that $\K^2(A)$ is a triangulated category.

\s Let $\T$ be the full subcategory in $\K^2(\A)$ consisting of $2$-periodic complexes $\Xb$ with $H^*(\Xb) = 0$. Then $\T$ is a thick subcategory in $\K^2(\A)$. The Verdier quotient $\D^2(\A) = \K^2(\A)/\T$ is the derived category of $2$-periodic $\A$-complexes. We also let $\K^2_{fg}(\A)$ denote the thick subcategory $\K^2(\A)$ consisting of $2$-periodic complexes $\Xb$ with $H^i(\Xb)$ finitely generated $A$-module for all $i \in \Z$. The Verdier quotient $\D^2_{fg}(\A) = \K^2_{fg}(\A)/\T$ is the derived category of $2$-periodic $\A$-complexes with finitely generated cohomology.

\section{Some construction on $2$-Periodic complexes}
In this section $A$ is a Noetherian ring.
In this section we describe some constructions that we need. For convenience we work with $\C^2(\Mod A)$, although some of the constructions hold more generally.

\s \label{dual} Let $\Xb \in \C^2(\Mod A)$ and let $M$ be an $A$-module. By $\Hom_A(\Xb, M)$ we mean a complex with
$\Hom_A(\Xb, M)^n = \Hom_A(\Xb^{-n}, M)$. The differentials are given as if $f \in \Hom_A(\Xb, M)^n $ then
$$ \partial(f) = (-1)^{n+1}f \circ \partial_{\Xb, -n-1}. $$
Clearly $\Hom_A(\Xb, M) \in \C^2(\Mod A)$. We will primarily use this construction when $M = A$ or when $M = E$, the injective hull of $A/\m$ (when $(A,\m)$ is local).

\s \label{ctensor} Let $\Xb, \Yb \in \C^2(\Mod A)$. By $\Zb = \Xb\ctensor \Yb$ (the $2$-periodic tensor of $\Xb$ with $\Yb$) we mean a complex with
\begin{enumerate}
  \item $\Zb^{2n}= \Zb^0 = (\Xb^0 \otimes_A \Yb^0) \oplus (\Xb^1 \otimes_A \Yb^1)$ for all $n \in \Z$.
  \item  $\Zb^{2n+1}= \Zb^1 = (\Xb^0 \otimes_A \Yb^1) \oplus (\Xb^{1} \otimes_A \Yb^0)$ for all $n \in \Z$.
  \item $\partial(x\otimes y) = \partial(x)\otimes y + (-1)^{|x|}x\otimes \partial(y)$.
\end{enumerate}
Here $|x| = 0 $ if $x \in \Xb^0$ and $|x| = 1 $ if $x \in \Xb^1$. It is readily verified that $\Zb$ is complex. Clearly it is circular.

\s\label{chom} Let $\Xb, \Yb \in \C^2(\Mod A)$. By $\Zb = \cHom(\Xb, \Yb)$ (the $2$-periodic Hom of $\Xb$ with $\Yb$) we mean a complex with
\begin{enumerate}
  \item $\Zb^{2n}= \Zb^0 = \Hom_A(\Xb^0, \Yb^0) \oplus \Hom_A(\Xb^1, \Yb^1)$ for all $n \in \Z$.
  \item  $\Zb^{2n+1}= \Zb^1 =  \Hom_A(\Xb^0, \Yb^1) \oplus \Hom_A(\Xb^1, \Yb^0)$  for all $n \in \Z$.
  \item $\partial(f) = \partial_{\Yb}\circ f - (-1)^{|f|}f \circ \partial_{\Xb}$.
\end{enumerate}
It is readily verified that $\Zb$ is infact a $2$-periodic complex. It is readily verified that
$$ H^0(\cHom(\Xb, \Yb)) = \Hom_{\K^2(\Mod A)}(\Xb, \Yb). $$

\s\label{hom-tensor}  Let $\Xb, \Yb \in \C^2(\Mod A)$. We have a natural map
\begin{align*}
 \delta_{\Xb, \Yb} \colon \Yb \ctensor \Hom_A(\Xb, A) &\rt \cHom(\Xb, \Yb) \\
 y\otimes \phi  &\rt     y\ctensor \phi
\end{align*}
Here $ y\ctensor \phi(x) = \phi(x)y$. It is readily verified that if $\Yb \in \C^2(\mmod(A))$ and $\Xb \in \C^2(\proj A)$ then
$\delta_{\Xb, \Yb}$ is an isomorphism.
\section{Minimal $2$-periodic complexes}
In this section we assume $(A,\m)$ is a Noetherian local ring. We say $\Xb \in \C^2(\proj A)$ is minimal if $\partial(\Xb) \subseteq \m \Xb$.
We first need a large class of complexes (whose image) in $\K^2(\proj A)$ is zero.
\begin{lemma}\label{zero-c}
\begin{enumerate}[\rm (1)]
  \item Let $\Wb$ be the complex $\Wb^i = A$ for all $i$ and $\partial^{2i+1} = 1_A$ and $\partial^{2i} = 0$ for all $i$.. Then $\Wb$ is zero in $\K^2(\proj A)$.
  \item Let $\Wb$ be the complex $\Wb^i = A$ for all $i$ and  $\partial^{2i} = 1_A$ and $\partial^{2i + 1} = 0$ for all $i$.. Then $\Wb$ is zero in $\K^2(\proj A)$.
\end{enumerate}
\end{lemma}
\begin{proof}
(1)  Let $s_i \colon \Wb^{i} \rt \Wb^{i-1}$ be given by $1_A$ for all $i$. This gives a $2$-periodic homotopy of $1_{\Wb}$ with $0_{\Wb}$. It follows that $\Wb$ is zero in $\K^2(\proj A)$.

(2) This is similar to (1).
\end{proof}
The main result of this section is
\begin{theorem}
 \label{red-min} Let  $\Xb \in \C^2(\proj A)$. Then $\Xb = \Yb \oplus \Zb$ where $\Yb$ is minimal and $\Zb$ is a finite direct sum of complexes of the type given in Lemma \ref{zero-c}. In particular  $\Zb  = 0$ in $\K^2(\proj A)$.
\end{theorem}
\begin{proof}
We induct on $r(\Xb) = \rank \Xb^0 + \rank \Xb^1$. If $ r = 1$ then either $\Xb^0 = 0$ or $\Xb^1 = 0$. Then as $\Xb$ is $2$-periodic it follows that $\Xb$ is a minimal complex.
Now assume that $r \geq 2$ and the result is proved for all complexes $\Fb$ with $r(\Fb) < r$.
If $\Xb$ is minimal then we dont have to prove anything. So assume $\Xb$ is not minimal. Without loss of any generality we may assume $\partial^{-1}$ is not minimal. We may choose basis
$B_{-1} = \{ u_1, \ldots, u_m \}$ of $\Xb^{-1}$ and basis $B_0 = \{ v_1, \ldots, v_n \}$ of $\Xb^0$ such that $\partial^{-1}(u_1) = v_1$ and $\partial^{-1}(u_j) \in A v_2 + \cdots A v_n$ for $j \geq 2$.
Now note that $\partial^0(v_1) = \partial^0\circ \partial^{-1}(u_1) = 0$. Now for $j \geq 2$ let
$\partial^0(v_j) = a_{j}u_1 + t_j$ such that $t_j \in Au_2 + \cdots  + Au_m$. Then as $\partial^{1}(\partial^0(v_j)) = 0$ and as  $\partial^{1}(u_j) \in A v_2 + \cdots A v_n$ for $j \geq 2$
it follows that $a_j = 0$. It follows that $\Xb = \Yb \oplus \Wb$ where $\Wb$ is of the form \ref{zero-c}(1).
As $r(\Yb) < r(\Xb)$ we are done by induction.
\end{proof}

\begin{remark}
\label{min-nonzero} Let $\Xb \in \C^2(\proj A)$ be a minimal non-zero complex. Then it is not difficult to show that $\Xb \neq 0$ in $\K^2(\proj A)$.
\end{remark}

\s Let $\W$ to be the class of complexes which is isomorphic to one of the two types of complexes in \ref{zero-c}. We show that
\begin{proposition}\label{W-class}
Let $\Wb \in \W$. Then
\begin{enumerate}[\rm (1)]
\item
$\Wb[1] \in \W$.
\item
$\Wb^* = \Hom_A(\Wb, A) \in \W$.
\end{enumerate}
\end{proposition}
\begin{proof}
(1) We show that if $\Wb$ is of type (1) in \ref{zero-c} then $\Wb[1]$ is isomorphic to a complex $\Zb$ of type (2) in \ref{zero-c}. We note that $\Yb = \Wb[1] $ is the complex given by $\Yb^i = A$ for all $i$ and $\partial_Y^{2i} = -1_A$ and
$\partial_Y^{2i + 1} = 0$. Let $\Zb$ be of type (2) in \ref{zero-c}. Then $f \colon \Yb \rt \Zb$ is an isomorphism where $f_{2i} = 1_A$ and $f_{2i +1} = -1_A$.
A similar argument yields that if $\Wb$ is of type (2) in \ref{zero-c} then $\Wb[1]$ is isomorphic to a complex $\Zb$ of type (1) in \ref{zero-c}.

(2)  If $\Wb$ is of type (1) or (2) in \ref{zero-c} then it is straightforward to see that
 $\Wb^* \cong \Wb[1]$.
The result follows.
\end{proof}
\section{Krull-Schmidt property of $\K^2(\proj A)$}
In this section $(A, \m)$ is a Henselian local ring. Recall an $A$-linear additive category $\C$ is Krull-Schmidt if every object in $\C$  decomposes into a finite direct sum of objects having local endomorphism rings. In this section we prove Theorem \ref{ks}.

We first prove
\begin{theorem}
\label{c-mod}Let $(A,\m)$ be a Henselian local ring. Then $\C^2(\mmod(A))$ is Krull-Schmidt.
\end{theorem}
\begin{proof}
Set $\C =  \C^2(\mmod(A))$.
It is clear that $\C$  is an abelian category.  So it has split idempotents.

If $M$ is a finitely generated $A$-module then let $\mu(M) = \dim_{A/\m} M/\m M$ be its minimal number of generators. If $\Xb \in  \C$  then set $\mu(\Xb) = \mu(\Xb^1) + \mu(\Xb^0)$. An easy induction on $\mu(\Xb)$  yields that $\Xb$ decomposes into a finite direct sum of indecomposables.

Let $\Xb, \Yb  \in \C$. Then we have an inclusion of $A$-modules
\[
\Hom_{ \C}(\Xb, \Yb)  \subseteq \Hom_A(\Xb^0, \Yb^0)\times  \Hom_A(\Xb^1, \Yb^1).
\]
In particular $\Hom_{ \C}(\Xb, \Yb)$ is a finitely generated $A$-module.

Let $\Xb \in \C$ be indecomposable. As idempotents split in $\C$, the  finitely generated $A$-algebra $\End_\C(\Xb)$ has no idempotents, see \cite[1.1]{LW}. Let $J$ be the Jacobson radical of $\End_\C(\Xb)$. We note that  $\End_\C(\Xb)/J$ is  semi-simple,
see \cite[20.6]{L}.
As $A$ is Henselian it follows that idempotents in  $\End_\C(\Xb)/J$ can be lifted to  $\End_\C(\Xb)$. But $\End_\C(\Xb)$ has no idempotents. It follows that $\End_\C(\Xb)/J$ is a divison ring. So $\End_\C(\Xb)$ is local.  Thus $\C$ is a Krull-Schmidt category.
\end{proof}
We next show
\begin{lemma}
(with hypotheses as in \ref{c-mod}) $\C^2(\proj A)$ is a Krull-Schmidt category.
\end{lemma}
\begin{proof}
This follows from the fact that if $\Xb \in \C^2(\proj A)$ and $\Xb = \Yb \oplus \Zb \in \C^2(\mmod(A))$ then $\Yb, \Zb \in \C^2(\proj A)$.
\end{proof}
Next we give
\begin{proof}[Proof of Theorem \ref{ks}]
Let $\Xb \in \C^2(\proj A)$. Then by \ref{red-min} $\Xb = \Yb \oplus \Zb$ where $\Yb$ is minimal and $\Zb = 0$ in $\K = \K^2(\proj A)$. We note that direct summand of a minimal complex is minimal. Furthermore if $\Wb$ is a minimal complex then $\Wb \neq 0$ in $\K$.
Thus we can decompose any complex $\Xb$ in $\K$ as a finite direct sum of minimal complexes each of which is indecomposable in $\C^2(\proj A)$. Thus it is sufficient to prove the following assertion:

If $\Xb$ is an indecomposable  minimal complex in $\C^2(\proj A)$ then it is indecomposable in $\K$.

We note that $$ \Hom_\K(\Xb , \Xb) = \Hom_{\C^2(\proj A)}(\Xb , \Xb)/I. $$
where $I$ is the two sided ideal consisting of $2$-periodic null-homotopic maps from \\ $\Xb$ to $\Xb$.
We show that $I$ is contained in the Jacobson radical of \\ $\Hom_{\C^2(\proj A)}(\Xb , \Xb)$. This will prove the result.

Let $f \in I$. We want to show thar $1-gf$ is invertible for any \\ $g \in \Hom_{\C^2(\proj A)}(\Xb , \Xb)$. Set $h = gf \in I$. Then as $\Xb$ is minimal we get that the induced map
$1-h \colon \Xb^i/\m \Xb^i \rt \Xb^i/\m \Xb^i$ is the identity for all $i$. By Nakayama lemma it follows that $1-h \colon \Xb^i \rt \Xb^i$ is surjective (and hence an isomorphism) for all $i$. Thus $1-h$ is invertible in $\C^2(\proj A)$.
\end{proof}
Finally we show
\begin{proposition}\label{ks-f}
Let $(A,\m)$ be a Henselian local ring. Then $\K^2_f(\proj(A))$ is Krull-Schmidt.
\end{proposition}
\begin{proof}
This follows from the fact that if $\Xb \in \K^2_f(\proj A)$ and $\Xb = \Yb \oplus \Zb \in \K^2(\proj A)$ then $\Yb, \Zb \in \K^2_f(\proj A)$.
\end{proof}
\section{A lemma}
In this section $(A,\m)$ is a Noetherian local ring.
Let $\K(\proj A)$ denote the homotopy category (possibly unbounded) complexes of finitely generated free $A$-modules. Recall a complex $\Xb \in \K(\proj A)$ is said to be minimal if
$\partial(\Xb) \subseteq \m \Xb$. The goal of this section is to prove the following result:
\begin{lemma}
    \label{mod-2} Assume  $\Fb \in \K(\proj A)$ is a minimal complex. Suppose there exists an isomorphism $\phi \colon \Fb[-2] \rt \Fb$ in $\K(\proj A)$. Then there exists a minimal $2$-periodic complex $\Xb \in \K(\proj A)$  and an isomorphism  $f \colon \Xb \rt \Fb$ in $\C(\proj A)$.
\end{lemma}
We will need the following easily proved result:
\begin{proposition}
  \label{min-iso} Let $\Xb, \Yb \in \K(\proj A)$ be minimal complexes and let $\phi \colon \Xb \rt \Yb$ be an isomorphism in $\K(\proj A)$. Then $\phi_i \colon \Xb^i \rt \Yb^i$ are isomorphism's for all $i \in \Z$. So $\phi$ is an isomorphism in $\C(\proj A)$.
\end{proposition}
We now give
\begin{proof}[Proof of Lemma \ref{mod-2}]
Let $\alpha^n \colon \Fb^n \rt \Fb^{n+1}$ be the differential of $\Fb$. By \ref{min-iso} we get $\phi_n \colon \Fb^{n} \rt \Fb^{n+2}$ is an isomorphism for all $n \in \Z$.
As $\phi$  is a chain map we also have
\begin{equation*}
  \phi_i\circ \alpha^{i-1} = \alpha^{i+1} \circ  \phi_{i-1} \quad \text{for all $i \in \Z$}. \tag{*}
\end{equation*}
As $\phi_i$ are isomorphisms for all $i$ we also get from (*)
\begin{equation*}
  \alpha^{i-1} \circ \phi_{i-1}^{-1} = \phi_i^{-1} \circ \alpha^{i+1} \quad \text{for all $i \in \Z$}. \tag{**}
\end{equation*}

We first construct $\Xb$. Set $\Xb^{2n} = \Fb^2$ and $\Xb^{2n + 1} = \Fb^1$ for all $n \in \Z$. The differentials on $\Xb$ are defined as follows: set $\partial^{2n+1} = \alpha^1$ and $\partial^{2n} = \alpha^0 \circ \phi_0^{-1}$ for all $n \in \Z$.
We show $\Xb$ is indeed a complex. We note
$$\partial^{2n+1} \circ \partial^{2n} = \alpha^1\circ (\alpha^0 \circ \phi_0^{-1}) = (\alpha^1 \circ \alpha^0) \circ \phi_0^{-1} = 0. $$
Next we have
$$ \partial^{2n} \circ \partial^{2n-1} = (\alpha^0 \circ \phi_0^{-1}) \circ \alpha^1 =  (\phi_1^{-1} \circ \alpha^2)\circ \alpha^1 = \phi_1^{-1} \circ ( \alpha^2\circ \alpha^1) = 0.$$
Here the second equality follows from (**) with $i = 1$. Thus $\Xb$ is indeed a complex. By construction $\Xb$ is a $2$-periodic complex. Also note that $\Xb$ is a  minimal complex.

We construct a chain map $f \colon \Xb \rt \Fb$. Define
\begin{enumerate}
  \item $f_0 = \phi_0^{-1}$, $f_1 = 1_{\Fb^1}$ and $f_2 = 1_{\Fb^{2}}$.
  \item For $n \geq 1$ define  $f_{2n +1} \colon \Xb^{2n+1} \rt \Fb^{2n+1}$ as $f_{2n+1} = \phi_{2n-1} \circ \phi_{2n-3} \circ \cdots \circ \phi_1$.
  \item For $n \geq 2$ define $f_{2n} \colon \Xb^{2n} \rt \Fb^{2n}$ as $f_{2n} = \phi_{2n-2} \circ \phi_{2n-4} \circ  \cdots \circ \phi_2$.
  \item For $n \leq 0$ define $f_{2n -1} \colon \Xb^{2n-1} \rt \Fb^{2n-1}$ as $f_{2n-1} = \phi_{2n-1}^{-1}\circ \phi_{2n + 1}^{-1} \circ \cdots \circ \phi_{-1}^{-1}$.
  \item For $n \leq -1$ define $f_{2n} \colon \Xb^{2n} \rt \Fb^{2n}$ as $f_{2n} = \phi_{2n}^{-1} \circ \phi_{2n +2}^{-1} \circ  \cdots \circ \phi_{0}^{-1}$.
\end{enumerate}
We note that $f_n$ are all isomorphisms. Thus it suffices to show $f$ is a chain map.
We have to show
\begin{equation*}
  \alpha^{n-1} \circ f_{n-1} = f_n \circ \partial^{n-1} \quad \text{for all $n \in \Z$}. \tag{$\dagger$}
\end{equation*}
Clearly $(\dagger)$ holds for $n = 1, 2$.
We show it holds for all $n \geq 3$ by induction on $n$.\\
For $n = 3$ we note that
\begin{align*}
  f_3 \circ \partial^2  &= \phi_1 \circ (\alpha^0 \circ \phi_0^{-1}) \\
   &= (\phi_1 \circ \alpha^0) \circ \phi_0^{-1}  \\
  &= (\alpha^2 \circ \phi_0) \circ \phi_0^{-1}, \quad \text{from (*) with $i = 1$} \\
   &= \alpha^2 \\
   & = \alpha^2 \circ f_2.
\end{align*}
For $n = 4$ we note that
\begin{align*}
  f_4\circ \partial^3 &= \phi_2 \circ \alpha^1 \\
  &= \alpha^3 \circ \phi_1, \quad \text{from (*) with $i = 2$} \\
   &= \alpha^3 \circ f_3.
\end{align*}
We now assume the  result for all $n $ with $1 \leq n\leq n_0$ (with $n_0 \geq 4$) and prove the result for $n_0 + 1$. We have to consider the cases $n_0+1$ is even or odd separately.
We first consider the case when $n_0 + 1 = 2m$. We have
\begin{align*}
  f_{2m} \circ \partial^{2m-1} &= (\phi_{2m-2} \circ \phi_{2m-4} \circ \cdots \circ \phi_0) \circ \alpha_1 \\
   &= \phi_{2m-2} \circ (\phi_{2m-4} \circ \cdots \circ \phi_0 \circ \alpha_1) \\
   &= \phi_{2m-2} \circ ( \alpha^{2m-3} \circ \phi_{2m-5} \circ \cdots \circ \phi_1),\quad \text{by induction hypothesis} \\
  &= (\phi_{2m-2} \circ  \alpha^{2m-3}) \circ (\phi_{2m-5} \circ \cdots \circ \phi_1) \\
  &= (\alpha^{2m-1} \circ \phi_{2m-3}) \circ (\phi_{2m-5} \circ \cdots \circ \phi_1) \quad \text{from (*) with $i = 2m-2$} \\
  &= \alpha^{2m-1} \circ f_{2m-1}.
\end{align*}
Next we consider the case when $n_0 + 1 = 2m + 1$.
We have
\begin{align*}
  f_{2m+1} \circ \partial^{2m} &= ( \phi_{2m-1} \circ \phi_{2m-3}\circ \phi_{2m-5} \circ \cdots \circ \phi_1) \circ \alpha^0 \circ \phi_0^{-1} \ \\
   &= \phi_{2m-1} \circ (\phi_{2m-3} \circ \phi_{2m-5} \circ \cdots \circ \phi_1 \circ \alpha^0 \circ \phi_0^{-1})\\
  &=  \phi_{2m-1} \circ (\alpha^{2m-2} \circ \phi_{2m-4} \circ \cdots \circ \phi_0 ),\quad \text{by induction hypothesis} \\
   &= (\phi_{2m-1} \circ \alpha^{2m-2}) \circ \phi_{2m-4} \circ \cdots \circ \phi_0 \\
  &=  (\alpha^{2m} \circ \phi_{2m-2}) \circ \phi_{2m-4} \circ \cdots \circ \phi_0  \quad \text{from (*) with $i = 2m-1$} \\
  &= \alpha^{2m} \circ f_{2m}.
\end{align*}
Thus $(\dagger)$ holds for $n \geq 1$. We show $(\dagger)$ holds for $n \leq 0$ by induction on $n$.
For $n = 0$ we have
\begin{align*}
  f_0 \circ \partial^{-1} &= \phi_0^{-1} \circ \alpha^1 \\
   &= \alpha^{-1} \circ \phi_{-1}^{-1}, \quad \text{from (**) with $i = 0$},  \\
   &= \alpha^{-1} \circ f_{-1}.
\end{align*}
For $n = -1$ we have
\begin{align*}
  f_{-1}\circ \partial^{-2} &= \phi_{-1}^{-1} \circ (\alpha_0 \circ \phi_0^{-1}) \\
  &= (\phi_{-1}^{-1} \circ \alpha_0) \circ \phi_0^{-1}  \\
  &= (\alpha^{-2} \circ \phi_{-2}^{-1}) \circ \phi_0^{-1} \quad \text{from (**) with $i = -1$}. \\
   &= \alpha^{-2} \circ f_{-2}.
\end{align*}
We now assume that $(\dagger)$ holds for all $n \geq n_0$ where $-2 \geq n_0$ and prove the result for $n_0 - 1$. We have to consider the cases $n_0 - 1$ is even or odd separately.
We first consider the case when $n_0 - 1 = 2m$. We have
\begin{align*}
  f_{2m + 1} \circ \partial^{2m} &=  (\phi_{2m+1}^{-1} \circ \phi_{2m+3}^{-1} \circ \cdots \circ \phi_{-1}^{-1}) \circ \alpha^0 \circ \phi_0^{-1} \\
  &= \phi_{2m+1}^{-1} \circ (\phi_{2m+3}^{-1} \circ \cdots \circ \phi_{-1}^{-1} \circ \alpha^0 \circ \phi_0^{-1}) \\
  &=  \phi_{2m+1}^{-1} \circ ( \alpha^{2m+2} \circ \phi_{2m +2}^{-1} \circ \cdots \circ \phi_0^{-1}), \quad \text{by induction hypotheses}\\
  &= (\phi_{2m+1}^{-1} \circ  \alpha^{2m+2}) \circ \phi_{2m +2}^{-1} \circ \cdots \circ \phi_0^{-1}\\
  &= (\alpha^{2m} \circ \phi_{2m}^{-1}) \circ  \phi_{2m +2}^{-1} \circ \cdots \circ \phi_0^{-1} \quad \text{from (**) with $i = 2m +1$}. \\
  &= \alpha^{2m} \circ f_{2m}.
\end{align*}
Next we consider the case when $n_0 - 1 = 2m -1$. We have
\begin{align*}
  f_{2m} \circ \partial^{2m-1} &= (\phi_{2m}^{-1} \circ  \phi_{2m +2}^{-1} \circ \cdots \circ \phi_0^{-1})\circ \alpha^{1}   \\
   &=  \phi_{2m}^{-1} \circ ( \phi_{2m +2}^{-1} \circ \cdots \circ \phi_0^{-1}\circ \alpha^{1}) \\
  &= \phi_{2m}^{-1} \circ( \alpha^{2m+1} \circ \phi_{2m+1}^{-1} \circ \phi_{2m+3}^{-1} \circ \cdots \circ \phi_{-1}^{-1}),  \  \text{by assumption}\\
  &=  (\phi_{2m}^{-1} \circ \alpha^{2m+1} )\circ \phi_{2m+1}^{-1} \circ \phi_{2m+3}^{-1} \circ \cdots \circ \phi_{-1}^{-1} \\
   &= \alpha^{2m-1} \circ \phi_{2m-1}^{-1} \circ \phi_{2m+1}^{-1} \circ \phi_{2m+3}^{-1} \circ \cdots \circ \phi_{-1}^{-1} \quad \text{by (**) with $i = 2m$} \\
  &= \alpha^{2m-1}\circ f_{2m-1}.
\end{align*}
Thus we have shown $(\dagger)$ for all $n\in \Z$. The result follows.
\end{proof}
\section{projective resolutions}
In this section we prove
\begin{theorem}\label{res}
Let $(A,\m)$ be a regular local ring.
Let $\Xb$ be a $2$-periodic complex with finitely generated cohomology. Then there exists $\Pb \in \C^2(\proj A)$ and a $2$-periodic quism $\phi \colon \Pb \rt \Xb$ in $\C^2(\Mod A)$.
\end{theorem}
The following is a first step to prove Theorem \ref{res}.
\begin{lemma}\label{bella}
 (with hypotheses as in Thoerem \ref{res})
 There exists  a minimal complex $\Qb \in \C(\proj A)$ and a quism $\xi \colon \Qb \rt \Xb$ in $\C(\Mod A)$.
\end{lemma}
We give a proof of Theorem \ref{res} by first assuming Lemma \ref{bella}
\begin{proof}[Proof of Theorem \ref{res}]
We may assume $H^*(\Xb) \neq 0$ (otherwise take $\Pb = 0$).
By Lemma \ref{bella} there exists  a minimal complex   $\Qb \in \C(\proj A)$ and a quism $\xi \colon \Qb \rt \Xb$ in $\C(\Mod A)$.
Notice we have a quism  $\xi[-2] \colon \Qb[-2] \rt \Xb[-2] = \Xb$.
Thus we have an invertible  map  $\phi = \xi^{-1} \circ \xi[-2]\colon \Qb[-2] \rt \Qb$ in $\D(\Mod A)$. By \ref{p-d} we get that $\phi \in \Hom_{\K(\Mod A)}(\Qb[-2], \Qb)$. As $\phi $ is invertible and $\Qb$ is minimal, by Lemma \ref{mod-2} there exists a minimal $2$-periodic complex $\Yb$ and a quism $f \colon \Yb \rt \Qb$. Set $g = \xi \circ f \colon  \Yb \rt \Xb$.
Note $g$ is a quism. As argued before we have  the map $\psi = g^{-1} \circ g[-2] \colon \Yb[-2] \rt \Yb$ is infact in  $\K(\Mod A)$. As $\Yb$ is $2$-periodic we have $\Yb = \Yb[-2]$.
We have $g \circ \psi = g[-2]$. So we have
\begin{equation*}
 g_i \circ \psi_i = g_{i-2} \quad \text{for all $i \in \Z$.} \tag{*}
\end{equation*}
 We also have by \ref{min-iso} that $\psi_i$ is in fact an isomorphism for all $i \in \Z$. \\
Consider $\Pb$ defined as follows:\\
$\Pb^n = \Yb^n$ for all $n \in \Z$. \\
Let $\partial_n$ be the differential in $\Yb$. Then the differential $\delta_n$ on $\Pb$ is defined as follows:\\
$\delta_{2n + 1} = \partial_{2n+1}$ and $\delta_{2n} = \partial_{2n} \circ  \psi_0^{-1}$ for all $n \in \Z$.

We first show $\Pb$ is a complex. Note $\delta_{2n+1} \circ \delta_{2n} = \partial_{2n+1} \circ  \partial_{2n} \circ \psi_{0}^{-1} = 0$. We also have
\begin{align*}
  \delta_{2n} \circ \delta_{2n -1}  &= \partial_{2n} \circ \psi_{0}^{-1} \circ \partial_{2n-1}  \\
   &=\partial_{0} \circ \psi^{-1}_0 \circ \partial_{-1}, \quad \text{as $\Yb$ is a $2$-periodic complex} \\
   &= \partial_{0} \circ \partial_{-1} \circ \psi^{-1}_{-1} \quad \text{as $\psi^{-1}$ is a chain map} \\
   &= 0.
\end{align*}
Clearly by construction $\Pb$ is a $2$-periodic complex. It is also minimal as $\Yb$ is minimal.

We define $p \colon \Pb \rt \Xb$ by $p_{2n} = g_0$ and $p_{2n + 1} = g_{-1}$. We show $p$ is a chain map. Clearly $p$ is circular.
Let $\alpha_n \colon \Xb^n \rt \Xb^{n+1}$ be the differential of $\Xb$.
Consider the diagram
\[
  \xymatrix
{
 \Pb^{-2}
\ar@{->}[r]^{\delta_{-2}}
\ar@{->}[d]^{g_0}
 & \Pb^{-1}
\ar@{->}[r]^{\delta_{-1}}
\ar@{->}[d]^{g_{-1}}
& \Pb^0
\ar@{->}[d]^{g_0}
&
\\
 \Xb^{-2}
\ar@{->}[r]^{\alpha_{-2}}
 & \Xb^{-1}
\ar@{->}[r]^{\alpha_{-1}}
& \Xb^{0}
\
 }
\]
As $g \colon \Yb \rt \Xb$ is a chain map we get that the righthand square above is commutative. We now note that
\begin{align*}
 g_{-1} \circ \delta_{-2} &=  g_{-1} \circ \partial_{-2} \circ \psi_0^{-1} \\
 &=  \alpha_{-2} \circ g_{-2} \circ \psi_0^{-1} \quad \text{as $g \colon \Yb \rt \Xb$ is a chain map.} \\
 &= \alpha_{-2} \circ g_0 \quad \text{ from (*) with $i = 0$}.
\end{align*}
We now show that $p$ induces an isomorphism in cohomology. As the degree $0, -1$ maps of $p$ are same as that of $g$ and $g$ is a quism, it suffices to prove that the cycles $Z^{i}(\Pb) = Z^{i}(\Yb)$ for $i = 0, -1$ and the boundaries $B^i(\Pb) = B^i(\Yb)$ $i = 0, -1$. Clearly $Z^{-1}(\Pb) = Z^{-1}(\Yb)$. Furthermore as $\psi_0^{-1}$ is an isomorphism we get that $B^{-1}(\Pb) = B^{-1}(\Yb)$. Clearly $B^{0}(\Pb) = B^0(\Yb)$. Also as $\psi^{-1}_0$ is an isomorphism  and $\Yb$ is cyclic we get
$$Z^0(\Pb) = \ker \partial_{-2} \circ \psi_0^{-1} = \ker \partial_{0} \circ \psi_0^{-1}  = \ker \partial_0 = Z^0(\Yb).$$
The result follows.
\end{proof}
We now give
\begin{proof}[Proof of Lemma \ref{bella}]
We will assume $H^*(\Xb) \neq 0$ otherwise there is nothing to show (we may take $\Qb = 0$).
Let $\partial_n \colon \Xb^n \rt \Xb^{n+1}$ be the differential of $\Xb$
Let $\Yb$ be the complex
$$ \cdots \rt \Xb^{-n}\rt  \cdots \rt \Xb^{-2} \rt \Xb^{-1} \rt \Xb^0 \xrightarrow{\partial_0} \image \partial_0 \rt 0. $$
Let $\Pb \in \C(\proj A)$ be a  minimal projective resolution of $\Yb$, i.e., we have a quism $p \colon \Pb \rt \Yb$.
Furthermore note that as $H^i(\Pb) = 0$ for $i > 0$ and as $\Pb$ is a minimal complex it follows that $\Pb^i = 0$ for $i > 0$.

Using the fact that $\Xb$ is a $2$-periodic complex we have a chain map $g \colon \Yb \rt \Yb[-2]$,
\[
  \xymatrix
{
\cdots
 &\Xb^{0}
\ar@{->}[r]^{ \partial_0}
\ar@{->}[d]^{1}
 & \image \partial_0
\ar@{->}[r]
\ar@{->}[d]^{j}
& 0
\ar@{->}[d]
& \
& \
&
\\
\cdots
 &\Xb^{-2}
\ar@{->}[r]^{\partial_{-2}}
 & \Xb^{-1}
\ar@{->}[r]
& \Xb^{0}
\ar@{->}[r]
& \image \partial^0
\ar@{->}[r]
&0
\
 }
\]
(here $j$ is an inclusion map). Note we are using the fact that $\Xb$ is $2$-periodic.  It follows that we have a chain map $f \colon \Pb \rt \Pb[-2]$ so that  we have the following  commutative diagram in $\K(\Mod A)$.
\[
  \xymatrix
{
 \Pb
\ar@{->}[r]^{p}
\ar@{->}[d]^{f}
 & \Yb
\ar@{->}[d]^{g}
&
\\
 \Pb[-2]
\ar@{->}[r]^{p[-2]}
 & \Yb[-2]
\
 }
\]
We note that $H^i(g) \colon H^i(\Yb) \rt H^i(\Yb[-2])$ is an isomorphism for $i \leq 0$. It follows that $H^i(f) \colon H^i(\Pb) \rt H^i(\Pb[-2])$ is an isomorphism for $i \leq 0$. We note that $\cone(f)$ is bounded above and $H^i(\cone(f)) =0 $ for $i \leq -1$. As $A$ is regular we have an isomorphism $\phi \colon \cone(f) \cong \Zb$ in $\K(\proj A)$ where $\Zb \in \K^b(\proj A)$ is a bounded complex of projective modules.

\textit{Claim-1:}  The maps $f_n$ are isomorphisms for $n \ll 0$. \\
We have a commutative diagram in $\K(\Mod A)$:
\[
  \xymatrix
{
 \Pb
\ar@{->}[r]^{ f}
\ar@{->}[d]^{\psi}
 & \Pb[-2]
\ar@{->}[r]^{i}
\ar@{->}[d]^{j}
& \cone(f)
\ar@{->}[d]^\phi
\ar@{->}[r]
& \Pb[1]
\ar@{->}[d]^{\psi[1]}
\\
 \cone(\phi \circ i [-1])
\ar@{->}[r]^{\delta}
 & \Pb[-2]
\ar@{->}[r]^{\phi \circ i}
& \Zb
\ar@{->}[r]
& \cone(\phi \circ i[-1])[1]
\
 }
\]
Here $j$ is the identity map. Note we have taken cone of $\phi \circ i [-1]$ and not of $\phi \circ i$ in the bottom row. Note $\psi$ exists because $\K(\Mod A)$ is a triangulated category. Furthermore as $j$ and $\phi$ are isomorphism we get $\psi$ is an isomorphism. Furthermore as $\Zb$ is a bounded complex  (and as $\Pb[-2]$ is a minimal complex) we get that the differentials in $\cone(\phi \circ i[-1])$ are minimal for $n \ll 0$ and $\delta_n$ is an isomorphism for $n \ll 0$. As $\psi$ is an isomorphism, it follows from
the fact that $\Pb$ is a minimal complex  and the differentials in $\cone(\phi \circ i[-1])$ are minimal for $n \ll 0$, that $\psi_n$ is an isomorphism for $n \ll 0$.
Note $f$ is homotopic to $\theta = j^{-1} \circ \delta \circ \psi$ and as $\Pb$ is a minimal complex and $\theta_n$ is an isomorphism for $n \ll 0$ it follows that $f_n$ is surjective for $n \ll 0$. As $\theta_n $ is an isomorphism for $n \ll 0$ it follows that $\rank \Pb^n =  \rank \Pb^{n-2}$ for $n \ll 0$. It follows that $f_n$ is an isomorphism for $n \ll 0$. Thus claim 1 is proved.

Consider the direct limits of the maps
\begin{enumerate}
  \item $$ \Yb \xrightarrow{g} \Yb[-2] \xrightarrow{g[-2]} \Yb[-4] \rt \cdots \Yb[-2n] \xrightarrow{g[-2n]} \Yb[-2n -2] \rt \cdots.$$
  and
  \item
   $$ \Pb \xrightarrow{f} \Pb[-2] \xrightarrow{f[-2]} \Pb[-4] \rt \cdots \Pb[-2n] \xrightarrow{f[-2n]} \Pb[-2n -2] \rt \cdots.$$
\end{enumerate}
It is clear that the first direct limit is $\Xb$. Also note if $r \in \Z$ is fixed then  by Claim-1 the maps $ \Pb[-2n]^r \xrightarrow{f[-2n]_r} \Pb[-2n -2]^r$ are isomorphisms for $n \gg 0$. it follows that $\Qb = \lim \Pb[-2n] \in \K(\proj A)$.

We have an exact sequence
 $$  0 \rt \bigoplus_{n \geq 0} \Yb[-2n] \xrightarrow{\pi_{\Yb}} \bigoplus_{n \geq 0}\Yb[-2n] \rt \left(\lim \Yb[-2n] = \Xb\right) \rt 0,$$
 where
 $$ \pi_{\Yb} =
 \begin{bmatrix}
   1_{\Yb} & 0 & 0 & 0 & \cdots\\
   -g & 1_{\Yb[-2]} & 0 & 0 & \cdots \\
   0 & -g[-2] & 1_{\Yb[-4]}& 0 & \cdots \\
   0 & 0 & -g[-4]&1_{\Yb[-6]} & \cdots \\
   \cdots & \cdots & \cdots & \cdots & \cdots
 \end{bmatrix}
$$

In $\D(\Mod A)$ we have a triangle
$$ \bigoplus_{n \geq 0} \Yb[-2n] \xrightarrow{\pi_{\Yb}} \bigoplus_{n \geq 0}\Yb[-2n] \rt  \Xb  \rt \left(\bigoplus_{n \geq 0} \Yb[-2n]. \right)[1]. $$
Similarly we have a triangle in $\D(\Mod A)$,
$$ \bigoplus_{n \geq 0} \Pb[-2n] \xrightarrow{\pi_{\Pb}} \bigoplus_{n \geq 0}\Pb[-2n] \rt  \Qb  \rt \left(\bigoplus_{n \geq 0} \Pb[-2n]. \right)[1]. $$
 We have the following  commutative diagram in $\K(\Mod A)$ (and so in $\D(\Mod A)$).
\[
  \xymatrix
{
 \Pb[-2n]
\ar@{->}[r]^{p[-2n]}
\ar@{->}[d]^{f[-2n]}
 & \Yb[-2n]
\ar@{->}[d]^{g[-2n]}
&
\\
 \Pb[-2n -2]
\ar@{->}[r]^{p[-2n -2]}
 & \Yb[-2n -2]
\
 }
\]
We note that $p[-2n]$ is an isomorphism in $\D(\Mod A)$. So we have a commutative diagram in $\D(\Mod A)$,
\[
  \xymatrix
{
 \bigoplus_{n \geq 0} \Pb[-2n]
\ar@{->}[r]^{ \pi_{\Pb}}
\ar@{->}[d]^{\psi = \bigoplus_{n \geq 0}p[-2n]}
 & \bigoplus_{n \geq 0} \Pb[-2n]
\ar@{->}[r]
\ar@{->}[d]^{\psi}
& \Qb
\ar@{->}[d]^\phi
\ar@{->}[r]
& \left(\bigoplus_{n \geq 0} \Pb[-2n]\right)[1]
\ar@{->}[d]^{\psi[1]}
\\
 \bigoplus_{n \geq 0} \Yb[-2n]
\ar@{->}[r]^{\pi_{\Yb}}
 & \bigoplus_{n \geq 0} \Yb[-2n]
\ar@{->}[r]
& \Xb
\ar@{->}[r]
& \left(\bigoplus_{n \geq 0} \Yb[-2n]\right)[1]
\
 }
\]
The first square is commutative.
We note that $\phi$ exists since $\D(\Mod A)$ is a triangulated category. Note $\psi$ is an isomorphism. So $\phi$ is an isomorphism in $\D(\Mod A)$.
The result follows from Theorem \ref{p-d}.
\end{proof}

\section{Proof of Theorem \ref{first}}
In this section we give a proof of Theorem \ref{first}. Throughout  this section $(A,\m)$ is a regular local ring.  We first show
\begin{proposition}
\label{acyclic-c}
Let $\Pb \in \K^2(\proj A)$. If $H^*(\Pb) = 0$ then $\Pb = 0$.
\end{proposition}
\begin{proof}
By \ref{red-min}, we have $\Pb = \Xb \oplus \Yb$ in $\C^2(\proj A)$ where $\Xb$ is a minimal complex and $\Yb = 0$ in $\K^2(\proj A)$. So it suffices to show that if $\Xb$ is a minimal complex in $\C^2(\proj A)$ then $H^*(\Xb) \neq 0$. Suppose if possible $H^*(\Xb) =0$ then
$$  \cdots \Xb^{-n} \rt \cdots \Xb^{-1} \rt \Xb^{0} \rt \image \partial_0 \rt 0$$
is an infinite minimal resolution of $\image \partial_0$. This contradicts the fact that $A$ is regular local.
\end{proof}
As a consequence we have
\begin{corollary}
  \label{q-c} Let $\Pb, \Qb \in \K^2(\proj A)$. If $f \colon \Pb \rt \Qb $ is a $2$-periodic quism then $f$ is invertible in $\K^2(\proj A)$.
\end{corollary}
\begin{proof}
  We have a triangle $\Pb \xrightarrow{f} \Qb \rt \cone(f) \rt \Pb[1]$ in $\K^2(\proj A)$. As $f$ is a quism it follows that $H^*(\cone f) = 0$. So by \ref{acyclic-c} we get that $\cone(f) = 0$
  in $\K^2(\proj A)$. It follows that $f$ is an isomorphism in $\K^2(\proj A)$.
\end{proof}

\s Let $\K^2_{fg}(\Mod A)$ be the homotopy category of $2$-periodic complexes with finitely generated cohomology, i.e., a $2$-periodic complex $\Xb$  is in  $\K^2_{fg}(\Mod A)$ if $H^i(\Xb)$ is finitely generated $A$-module for all $i$. Let $\D^2_{fg}(\Mod A)$ be its derived category.
We first show
\begin{lemma}
  \label{c-left} Let $\Pb \in \K^2(\proj A)$ and let $\Xb \in \K^2_{fg}(\Mod A)$. If $s \colon \Xb \rt \Pb$ is a $2$-periodic quism then there exists a $2$-periodic quism $t \colon \Pb \rt \Xb$ such that
  $s\circ t = 1_{\Pb}$.
\end{lemma}
\begin{proof}
  By Theorem \ref{res} there exists a $2$-periodic quism  $\eta \colon \Qb \rt \Xb$ be  with $\Qb \in \C(\proj A)$. Then $g = s \circ \eta \colon \Qb \rt \Pb$ is a quism and so invertibe in $\K^2(\proj A)$. Set $t = \eta  \circ g^{-1}$. The result follows.
\end{proof}

Next we show
\begin{theorem}
 \label{p-d-c} Let $\Pb \in \K^2(\proj A)$ and let $\Xb \in \K^2_{fg}(\Mod A)$. Then the natural map
 \[
 \theta \colon  \Hom_{\K^2_{fg}(\Mod A)}(\Pb, \Xb) \rt  \Hom_{\D^2_{fg}(\Mod A)}(\Pb, \Xb)
 \]
is an isomorphism.
\end{theorem}
\begin{proof}
Set $\D = \D^2_{fg}(\Mod A)$.
 Let $f \circ s^{-1} \colon \Pb \xleftarrow{s} \Zb \xrightarrow{f} \Xb$ be a left fraction. Note $s$ is a quism. By \ref{c-left} there exists a $2$-periodic quism $t \colon  \Pb \rt \Zb$ such that $s\circ t = 1_{\Pb}$. Then check that $f\circ s^{-1} = f\circ t \in   \Hom_{\D}(\Pb, \Xb)$. Thus $\theta$ is surjective.

 Let $f \colon \Pb \rt \Xb$ be such that $\theta(f) = 0$. Then by \cite[2.1.26]{N}, there exists a $2$-periodic quism $s \colon \Zb \rt \Pb$ such that $f \circ s = 0$. By \ref{c-left} there exists a $2$-periodic quism $t \colon  \Pb \rt \Zb$ such that $s\circ t = 1_{\Pb}$.
 Then $f = f \circ s \circ t = 0$. So $\theta$ is injective.
\end{proof}
As a consequence we give a proof of Theorem \ref{first}. We restate it for the convenience of the reader.
\begin{theorem}
\label{fg=equi} The natural map $\eta \colon \K^2(\proj A) \rt \D^2_{fg}(\Mod A)$ is an equivalence of triangulated categories.
\end{theorem}
\begin{proof}
  By Theorem \ref{p-d-c} it follows that $\eta$ is fully faithful. By Theorem \ref{res} it is dense. The result follows.
\end{proof}

We will need the following result:
\begin{proposition}
\label{dual-fg}
Let $\Xb \in \K^2(\proj A)$.
\begin{enumerate}[\rm (1)]
  \item If $\Xb = 0$ in  $\K^2(\proj A)$ then $\Xb^* = 0$ in  $\K^2(\proj A)$
  \item If $H^i(\Xb)$ has finite length for all $i$ then $H^i(\Xb^*)$ has finite length for all $i$.
\end{enumerate}
\end{proposition}
\begin{proof}
(1) By \ref{red-min}, $\Xb = \Yb \oplus \Zb$ in $\C^2(\proj A)$ where $\Yb$ is a minimal complex and $\Zb = 0$ in $\K^2(\proj A)$.  Furthermore $\Zb$ is of the type $\W$ described in \ref{zero-c}.
As $\Xb = 0$ in $\K^2(\proj A)$  it follows that $\Yb = 0$. So $\Xb = \Zb$. By \ref{W-class} we get that $\Zb^* \in \W$. So $\Xb^* = 0$ in $\K^2(\proj A)$.

(2) Let $P \neq \m$ be a prime ideal in $A$. Note $A_P$ is regular local.  If $H^i(\Xb)$ has finite length for all $i$ then $H^i({\Xb}_{P}) \cong  H^i(\Xb)_P = 0$ for all $i$. By \ref{acyclic-c}
we get that ${\Xb}_P = 0$ in $\K^2(\proj A_P)$. By (1) we get that $({\Xb}_P)^* = 0$ in $\K^2(\proj A_P)$. Notice $({\Xb}_P)^* \cong (\Xb^*)_P$. It follows that $H^i(\Xb^*)$ has finite length for all $i$.
\end{proof}

\section{Preliminaries on AR-triangles and irreducible maps}

 The notion of AR-triangles in a triangulated Krull-Schmidt category was introduced by Happel \cite[Chapter 1]{Happel}.
In this section we recall the definition of AR-triangles, discuss a result from \cite{RV} on the existence of AR-triangles and discuss some facts about irreducible maps.

\s Let $\Cc$  be a Krull-Schmidt triangulated category with shift functor $ \sum$.

A triangle $ N \xrightarrow{f} E \xrightarrow{g} M \xrightarrow{h} \sum N$  in $\Cc$ is called a \emph{ right AR-triangle} (ending at $M$) if

   (RAR1) \  \ $M, N$ are indecomposable.

   (RAR2) \ \  $h \neq 0$.

   (RAR3) \  \ If $D$ is indecomposable then for every non-isomorphism $t \colon D \rt M$ we have $h\circ t = 0$.

Dually,  a triangle $\sum^{-1} M \xrightarrow{w} N \xrightarrow{f} E \xrightarrow{g} M $  in $\Cc$ is called a \emph{left AR-triangle} (starting at $N$) if

   (LAR1) \  \ $M, N$ are indecomposable.

   (LAR2) \ \  $w \neq 0$.

   (LAR3) \  \ If $D$ is indecomposable then for every non-isomorphism $t \colon N \rt D$ we have $t\circ w = 0$.

\begin{definition}
We say a Krull-Schmidt triangulated category $\Cc$ has AR-triangles if for any indecomposable $M \in \Cc$ there exists a right AR-triangle ending at $M$ and a left AR-triangle starting at $M$.
\end{definition}

\s In a fundamental work Reiten and Van den Bergh, \cite{RV}, related existence of AR-triangles to existence of a Serre functor.
Let $(A,\m)$ be a Noetherian local ring and let $E$ be the injective hull of $k = A/\m$.
Set $(-)^\vee = \Hom_A(-, E)$.
 Let $\Cc$ be a Hom-finite $A$-triangulated category.  By a right Serre-functor  on $\Cc$ we mean an additive functor $F \colon \Cc \rt \Cc$ such that we have isomorphism
 \[
 \eta_{C, D} \colon \Hom_{\Cc}(C, D) \rt \Hom_{\Cc}(D, F(C))^\vee,
 \]
 for any $C, D \in \Cc$ which are natural in $C$ and $D$. It can be shown that if $F$ is a right Serre functor
 then $F$ is fully faithful. If $F$ is also dense then we say $F$ is a Serre functor.
 We will use the following result:
 \begin{theorem}
 \label{RV-main}[see \cite[Theorem I.2.4]{RV}]
 Let $\Cc$ be a Hom-finite $A$-linear triangulated category. Then the following are equivalent
 \begin{enumerate}[\rm (i)]
 \item
 $\Cc$ has AR-triangles.
 \item
 $\Cc$ has a Serre-functor.
 \end{enumerate}
 \end{theorem}

 \begin{remark}
 In \cite{RV}, Theorem I.2.4 is proved for $k$-linear Hom-finite triangulated categories where $k$ is a field. However the same proof generalizes to our situation.
 \end{remark}

 \s \emph{Irreducible maps and AR-triangles:} First note that the notion of split \\  monomorphism and split epimorphism makes sense in any triangulated category. They are called section and retraction respectively in \cite{Happel}. Throughout $(A,\m)$ is a local Noetherian ring,  $\Cc$ is a Hom-finite $A$-triangulated category
 \s A map $f \colon X \rt Y$ in $\Cc$ is said to be irreducible if
 \begin{enumerate}
   \item $f$ is not a section or a retraction.
   \item if we have the following commutative diagram in $\Cc$
   \[
\xymatrix{
\
&U
\ar@{->}[dr]^{h}
 \\
X
\ar@{->}[rr]_{f}
\ar@{->}[ur]^{g}
&\
&Y
}
\]
 \end{enumerate}
 then either $g$ is a section or $h$ is a retraction.

 \s Let $s \colon  M \rt E \rt X \rt \sum M$ be an AR-triangle  ending in $X$. If $Z$ is indecomposable and $Z \rt X$ is an irreducible map then $Z$ is a direct summand of $E$. Similarly if $M \rt Z$ is an irreducible map then $Z$ is a direct summand of $E$, see \cite[Chapter 1, 4.3]{Happel}.

 \s Let $s \colon  M \rt E \xrightarrow{\phi} X \rt\sum M$ be an AR-triangle  ending in $X$. Let $Z$ be an indecomposable direct summnand of $E$. Say $E = Z \oplus Y$.   Denote
$\phi = (f,g)$ along this decomposition. Then $f \colon Z \rt M$ is an irreducible map
 (this can be proved in a similar manner to \cite[2.1.2]{Y}). An analogous  result holds for
 AR-triangles starting at $M$, see \cite[(2.1.2)', p. \ 14]{Y}.
 \s \label{irr} Assume $A/\m$ is an algebraically closed field. Let $Z$ be indecomposable.  Let $s \colon  M \rt E \rt X \rt \sum M$ be an AR-triangle  ending in $X$. Let $n$ be the number of copies of $Z$ in a decomposition of $E$ into indecomposables. Set $irr(Z, X) = n$. Then $irr(M, Z) = n$ (this can be proved in a similar manner to \cite[5.7]{Y}).
\section{Proof of Theorem \ref{ar}}
In this section we give:
\begin{proof}[Proof of Theorem \ref{ar}]
Let $\Xb, \Yb \in \K = \K^2_f(\proj A)$ and let $\D = \D^2_{fg}(\Mod A)$. Let $\mathbf{p} \colon \D \rt \K$ be the projective resolution functor. Then we have
\begin{align*}
  \Hom_\K(\Xb, \Yb)^\vee &= \left(H^0(\cHom_\K(\Xb, \Yb))\right)^\vee \\
   &\cong H^0\left( \Hom_A( \cHom_\K(\Xb, \Yb), E) \right)\\
   &\cong H^0\left( \Hom_A(\Yb\ctensor \Xb^*, E) \right) \ \\
  &\cong H^0\left(\cHom_\K(\Yb, (\Xb^*)^\vee) \right) \\
   &= \Hom_\K(\Yb, (\Xb^*)^\vee) \\
   &\cong\Hom_\mathcal{D}(\Yb, (\Xb^*)^\vee) \\
   &\cong \Hom_\mathcal{D}(\Yb, \mathbf{p}\left((\Xb^*)^\vee\right)) \\
   &\cong \Hom_\mathcal{K}(\Yb, \mathbf{p}\left((\Xb^*)^\vee\right))
\end{align*}
Thus we have an isomorphism
$$ \Hom_\K(\Xb, \Yb) \cong \left(\Hom_\mathcal{K}(\Yb, \mathbf{p}\left((\Xb^*\right))^\vee)\right)^\vee.$$
Clearly the above ismorphism is natural in $\Xb, \Yb$. Thus the functor $F$ defined as $F(\Xb^*) = \mathbf{p}\left((\Xb^*)\right)^\vee$ is a right Serre-functor.
It suffices to show $F$ is dense to conclude.

The functor $F$ is the composite of the following functors
\[
\K \xrightarrow{*} \K^{op} \xrightarrow{\Hom_A(-, E)} \K_f^c( add (E))\xrightarrow{\mathbf{p}}  \K
\]
It suffices to show each functor above is an equivalence. The functors  \\ $ * = \Hom_A(-,A)$ and  $\Hom_A(-, E)$ are clearly equivalences. It suffices to prove that $\mathbf{p}$ is an equivalence.

 We first prove that for any $\Xb \in \K$ there exists $\Yb \in  \K_f^c( add (E))$ and a quism $\Xb \rt \Yb$. Consider $\Zb = \Hom_A(\Xb, E)$ then $\Zb \in  \K_f^c( add (E))$. Let $\Wb \rt \Zb$ be a minimal
projective resolution of $\Zb$ which exists as $H^i(\Zb)$ has finite length for all $i$.
Then dualizing we obtain a quism $\Zb^\vee \rt \Wb^\vee = \Yb$. It remains to note that
 $\Zb^\vee \cong \Xb$. Thus we have a functor $\mathbf{i} \colon \K \rt \K_f^c( add (E))$
 (the injective resolution functor) which is clearly the inverse of $\mathbf{p}$.

 Thus $F$ is an equivalence and so $F$ is a Serre-functor. Hence $\K$ has AR-triangles
\end{proof}

\section{AR-quiver when $\dim A = 1$}
In this section $(A,\m)$ is a DVR with $\m = (x )$ and  $k =A/\m$ an algebraically closed field. In this section we describe the AR-quiver of $\K = \K^2_f(\proj A)$. We first need to describe all the indecomposables in $\K$.

\s Fix $j \geq 1$. Let $\Kb(j)$ be the complex defined as $\Kb(j)^n = A$ for all $n \in \Z$. Define $\partial^{2n+1}(t) = x^j t$ and $\partial^{2n} = 0$.
We first prove
\begin{proposition}\label{ind}
Let $\Xb \in \K$ be indecomposable. Then $\Xb \cong \Kb(j)$ or $\Xb \cong \Kb(j)[1]$ for some $j \geq 1$.
\end{proposition}
Before proving  the above result we need to recall the following result from  \cite[Lemma 7.6]{P}.
\begin{lemma}\label{stable}
Let $\Yb \in \K$. Then $\rank_A \Yb^i  = \rank \Yb^{i+1}$ for all $i \in \Z$.
\end{lemma}
We  now give
\begin{proof}[Proof of Proposition \ref{ind}]
By Lemma \ref{stable} it follows that $\Kb(j)$ (and hence \\ $\Kb(j)[1]$ ) are indecomposable for all $j \geq 1$.

Let $\Xb \in \K$ be indecomposable. We may assume $\Xb$ is minimal and $\partial^1 \neq 0$.
Also note that by Lemma \ref{stable} we have $\rank_A \Xb^i  = \rank \Xb^{i+1}$ for all $i \in \Z$.
 We consider the Smith form of $\partial^1$. Thus there exists basis $\{ u_1, \ldots, u_m \}$ of $\Xb^1$ and a basis $\{ v_1, \ldots, v_m \}$ of $\Xb^2$ with $\partial^1(u_j) = x^{a_j}v_j$ for $j = 1, \ldots, r$ and $\partial^1(u_j) = 0$ for $j > r$. It is elementary to see that $\partial^2(v_1)  = 0$ and $\partial^2(v_j) \in Au_2 + \cdots Au_m$. Thus $\Kb(a_1)$ is a direct summand of $\Xb$. As $\Xb$ is indecomposable it follows that $\Xb \cong \Kb(a_1)$.
\end{proof}

\begin{lemma}\label{irr-1}
There does not exist any irreducible maps from $\Kb(1)$ to itself.
\end{lemma}
\begin{proof}
Let $f \colon \Kb(1) \rt \Kb(1)$ be a $2$-periodic map. Let $f_i \colon  \Kb(1)^i \rt \Kb(1)^i$ be the components of $f$. Then note that $f_1 = f_2$ is multiplication by some $a_f$. Set $a_f = x^mu$. If $m \geq 1$ then note $f$ is null-homotopic (a hompotopy is given by $s_{2n} = x^{m-1}u$ and $s_{2n+1} = 0$). If $u \neq 0$ and $m = 0$ then note that
$f$ is an isomorphism. Thus any map $f \colon \Kb(1) \rt \Kb(1)$ is an isomorphism or zero. The result follows.
\end{proof}

We need the following computation of  the Auslander-Reiten translate of $\Kb(i)$.
\begin{lemma}\label{translate}
Fix $j \geq 1$. The Auslander-Reiten translate of $\Kb(i)$ is itself.
\end{lemma}
\begin{proof}
Let $F \colon \K \rt \K$ be the Serre-functor on $\K$. Then by \cite[1.2.3]{RV} the Auslander-Reiten translate of $\Kb(i)$ is given by $F[-1]$.

We note that $\Kb(i)^* \cong \Kb(i)[1]$. The cohomology of $(\Kb(i)^*)^\vee$ is $(A/(x^i)^\vee = A/(x^i)$ in degrees odd and zero in degrees even. It follows that $F(\Kb(i)) = \Kb(i)[1]$. The result follows.
\end{proof}

We now compute the AR-quiver of $\K$.
\begin{theorem}\label{quiver}
\begin{enumerate}[\rm (1)]
\item
The AR-triangles of $\K$ are as follows
\begin{align*}
\Kb(1) &\rt \Kb(2)\rt \Kb(1) \rt \Kb(1)[1], \\
\Kb(i) &\rt \Kb(i-1)\oplus \Kb(i+1) \rt \Kb(i) \rt \Kb(i)[1] \quad \text{for $i \geq 2$,} \\
\Kb(1)[1] &\rt \Kb(2)[1]\rt \Kb(1)[1] \rt \Kb(1), \\
\Kb(i)[1] &\rt \Kb(i-1)[1]\oplus \Kb(i+1)[1] \rt \Kb(i)[1] \rt \Kb(i) \quad \text{for $i \geq 2$.}
\end{align*}
\item
The AR-quiver of $\K$ is
\begin{align*}
\Kb(1) &\leftrightarrows \Kb(2) \leftrightarrows \Kb(3) \leftrightarrows \cdots \leftrightarrows \Kb(i) \leftrightarrows \cdots  \\
\Kb(1)[1] &\leftrightarrows \Kb(2)[1]  \leftrightarrows \Kb(3)[1] \leftrightarrows \cdots \leftrightarrows \Kb(i)[1] \leftrightarrows \cdots
\end{align*}
\end{enumerate}
\end{theorem}
\begin{proof}
The assertion (2) follows from (1). Also the third and fourth assertion of (1) follows by taking the shift operator $[1]$ and noting that $[2]$ is identity.

Let the Auslander-Reiten triangle ending  at $\Kb(1)$ be given as
\[
\Kb(1) \rt E\rt \Kb(1) \rt \Kb(1)[1]
\]
Taking cohomology we get that $H^1(E)  = 0$ and $\ell(H^0(E)) = 2$. It follows that
$E = \Kb(2)$ or $E = \Kb(1) \oplus \Kb(1)$. If the latter case holds then there is an irreducible map  $\Kb(1) \rt \Kb(1)$ which is not possible by Lemma \ref{irr-1}.

We prove the second assertion of (1) by induction on $i$.
Let the Auslander-Reiten triangle ending  at $\Kb(2)$ be given as
\[
\Kb(2) \rt E \oplus \Kb(1)\rt \Kb(2) \rt \Kb(2)[1]
\]
Taking cohomology we get that $H^1(E)  = 0$ and $\ell(H^0(E)) = 3$. It follows that
$E = \Kb(3)$ or $E = \Kb(2) \oplus \Kb(1)$ or $E = \Kb(1)^3$. The latter two cases are not possible as $irr(\Kb(1), \Kb(2)) = irr(\Kb(2),\Kb(1)) = 1$, see \ref{irr}.

We now assume that for all $i \leq r$ with $r \geq 2$ the AR-triangle ending at $\Kb(i)$ is given by
\[
\Kb(i) \rt \Kb(i-1) \oplus \Kb(i+1)\rt \Kb(i) \rt \Kb(i)[1].
\]
Let the AR-triangle ending at $\Kb(r+1)$ be given as
\[
\Kb(r+1) \rt E \oplus \Kb(r)\rt \Kb(r+1) \rt \Kb(r+1)[1]
\]
Taking cohomology we get that $H^1(E)  = 0$ and $\ell(H^0(E)) = r+2$.
We note that

(a) if $i < r$ and $\Kb(i)$ is a summand of $E$ then we have an irreducible map from $\Kb(r+1)$ to $\Kb(i)$ which is a contradiction.\\
(b) If $\Kb(r)$ is a direct summand of $E$ then $$ irr(\Kb(r+1), \Kb(r)) = irr(\Kb(r), \Kb(r+1)) \geq 2$$ which contradicts our inductive assumption that $irr(\Kb(r+1), \Kb(r)) = 1 $, see \ref{irr}.\\
(c)If $\Kb(r+1)$ is a direct summand of $E$ then note it follows that $\Kb(1)$ is also a summand of $E$, contradicting (a).\\
(d) As $\ell(H^0(E)) = r + 2$ note that $\Kb(i)$ is not a direct summand of $E$ for $i \geq r+3$.

Thus the only possible summand of $E$ is $\Kb(r+2)$ and as $\ell(H^0(E)) = \ell(\Kb(r+2)) = r+2$ it follows that $E = \Kb(r+2)$.
\end{proof}
 
\end{document}